\documentclass[reqno, 12pt]{article}

\pdfoutput=1

\usepackage{enumerate}
\usepackage{latexsym}
\usepackage[centertags]{amsmath}
\usepackage{amsfonts}
\usepackage{amssymb}
\usepackage{amsthm}
\usepackage{newlfont}
\usepackage{graphics}
\usepackage{color}
\usepackage{float}
\usepackage{diagbox}

\usepackage{booktabs}
\usepackage{algorithm}
\usepackage{algorithmicx}
\usepackage{algpseudocode}
\usepackage{url} 
\usepackage[pagebackref]{hyperref}
\usepackage{hyperref}
\usepackage{longtable}
\usepackage{rotating}
\usepackage{multirow}
\usepackage{extarrows}
\usepackage[sort,compress,numbers]{natbib}
\usepackage[utf8]{inputenc}
\usepackage{url}

\newtheorem{thm}{Theorem}[section]
\newtheorem{cor}[thm]{Corollary}
\newtheorem{lem}[thm]{Lemma}
\newtheorem{prop}[thm]{Proposition}

\theoremstyle{mydefinition}
\newtheorem{dfn}[thm]{Definition}
\theoremstyle{myremark}

\newtheorem{exa}[thm]{Example}

\allowdisplaybreaks[4]

\makeatletter
\let\c@algorithm\c@thm
\makeatother
\numberwithin{algorithm}{section}

\title{The Frobenius problem for a class of quotients of numerical semigroups}

\author{Feihu Liu
\\[2mm]
{\small School of Mathematical Sciences,}\\[-0.8ex]
{\small Capital Normal University, Beijing, 100048, P.R.~China}\\
{\small Email address: liufeihu7476@163.com}\\
}

\date{\today}

\begin{document}

\maketitle

\begin{abstract}
Given a numerical semigroup $S$ and a positive integer $p$, the quotient $\frac{S}{p}=\{x\in \mathbb{N} \mid px\in S\}$ also forms a numerical semigroup. In this paper, we first characterize the Ap\'ery set for a class of quotients of numerical semigroups.  
Under certain conditions, we then derive half-closed form formulas for their Frobenius number and genus.
Furthermore, for specific values of part parameters, we obtain explicit formulas for the Frobenius number of certain quotients of numerical semigroups.
\end{abstract}

\noindent
\begin{small}
\emph{2020 Mathematics subject classification}: Primary 11D07; Secondary 20M14.
\end{small}

\noindent
\begin{small}
\emph{Keywords}: Quotient of a numerical semigroup; Ap\'ery set; Frobenius number; Genus.
\end{small}

\section{Introduction}

Let $\mathbb{N}$ denote the set of all non-negative integers.
A subset $S$ of $\mathbb{N}$ is called a \emph{submonoid} if it contains $0$ and is closed under addition in $\mathbb{N}$.
When $\mathbb{N}\setminus S$ is finite, we call $S$ a \emph{numerical semigroup}.
Given a sequence $A = (a_1, a_2, \ldots, a_n)$ of positive integers, we denote by $\langle A \rangle$ the submonoid of $\mathbb{N}$ generated by $A$, which consists of all finite linear combinations of elements of $A$ with non-negative integer coefficients:
$$\langle A \rangle = \left\{ x_1a_1+x_2a_2+\cdots+x_na_n \mid x_i \in \mathbb{N}\quad \text{for}\quad 1 \leq i \leq n \right\}.$$
In \cite[Lemma 2.1]{J.C.Rosales}, it is shown that $\langle A\rangle$ forms a numerical semigroup if and only if $\gcd(A)=1$.
When this condition holds, we refer to $A$ as a \emph{system of generators} for the numerical semigroup $\langle A\rangle$.
Throughout this paper, we assume $\gcd(A)=1$.

This has led many researchers to investigate the following fundamental invariants:
1) The \emph{Frobenius number}, denoted by $F(\langle A\rangle)$, is the largest integer not contained in $\langle A\rangle$.
2) The \emph{genus} (or \emph{Sylvester number}), denoted by $g(\langle A\rangle)$, counts the number of positive integers not belonging to $\langle A\rangle$. In other words, 
$$F(\langle A\rangle)=\max \{\mathbb{N}\setminus \langle A\rangle \}\quad \text{and}\quad 
g(\langle A\rangle) = \# \{\mathbb{N}\setminus \langle A\rangle \}.$$
Given a sequence $A$ of positive integers, computing $F(\langle A\rangle)$ and $g(\langle A\rangle)$ is often referred to as the \emph{Frobenius Problem} or the \emph{Coin Exchange Problem}, see \cite{Ramrez Alfonsn}.
The determination of $F(\langle A\rangle)$ was shown by Ram\'irez Alfons\'in \cite{RamrezAlfonsn-Combin} to be NP-hard under Turing reduction.

Sylvester \cite{J.J.Sylvester1} established the following formulas 
\begin{align*}
F(\langle a_1,a_2\rangle) = a_1a_2 - a_1 - a_2 \quad \text{and} \quad g(\langle a_1,a_2\rangle) = \frac{1}{2}(a_1-1)(a_2-1).
\end{align*}
For the case when $n \geq 3$, Curtis \cite{F.Curtis} demonstrated that no closed-form expression of a certain type exists for $F(A)$. Nevertheless, numerous special cases have been investigated; we refer to \cite[Section 2; Section 3]{Ramrez Alfonsn} and \cite{Liu-Xin} for detailed results.

Let $p\in \mathbb{Z}^{+}$. Here $\mathbb{Z}^{+}$ denotes the set of all positive integers.
We define the \emph{quotient of $\langle A\rangle$ by $p$}, denoted by $\frac{\langle A\rangle}{p}$, to be the set consisting of all non-negative integers $x$ such that $px \in \langle A\rangle$; that is,
$$\frac{\langle A\rangle}{p} = \{ x \in \mathbb{N} \mid px \in \langle A\rangle \}.$$
It is easy to verify (see \cite[Section 5]{J.C.Rosales}) that $\frac{\langle A\rangle}{p}$ is a numerical semigroup, that $\langle A\rangle \subseteq\frac{\langle A\rangle}{p}$, and that $\frac{\langle A\rangle}{p}=\mathbb{N}$ if and only if $p\in\langle A\rangle$. The generators of the numerical semigroup $\frac{\langle A\rangle}{p}$ are discussed in \cite{Liu-BAMS}.

When $A = (a_1, a_2)$, a fundamental problem is to determine explicit formulas for the Frobenius number $F\left(\frac{\langle A\rangle}{p}\right)$ and the genus $g\left(\frac{\langle A\rangle}{p}\right)$. 
This case remains an open problem, as noted in \cite{MDelgado13}. Recent progress has been made in \cite{E.Cabanillas,F.Strazzanti}.
When $A$ is an almost arithmetic progression and $p$ is a positive divisor of $a_1$, the Frobenius number $F\left(\frac{\langle A\rangle}{p}\right)$ and the genus $g\left(\frac{\langle A\rangle}{p}\right)$ is studied in \cite{A.Adeniran,Liu-IJAC}.

In this paper, we consider the following class of quotients of numerical semigroups.
\begin{dfn}{\em \cite{LiuSemigroupF}}
Let $a,k\geq 2$, $d\in\mathbb{Z}\backslash \{0\}$, and $\gcd(a,d)=1$.
Define the sequences $B=(b_1,b_2,\ldots,b_k)$ and $H=(h_1,h_2,\ldots,h_k)$, where
\begin{align}
& b_1=1,\ b_{i+1}=s_ib_i+1\quad \text{and}\quad s_i\geq s_{i-1}\geq 1\quad \text{for}\quad 1\leq i\leq k-1,\label{BBcondition}
\\& h_i=ub_i+1, u\in \mathbb{Z}^{+}\quad \text{for}\quad 1\leq i\leq k.\label{HHcondition}
\end{align}
A family of numerical semigroups generated by $A$ is called a \emph{\texttt{GCNS} numerical semigroup} if
$$A=(a,Ha+dB)=(a,h_1a+db_1,h_2a+db_2,\ldots,h_ka+db_k),$$
with the additional requirement that $h_ia+db_i>1$ for $1\leq i\leq k$ when $d$ is a negative integer.
In the special case where $s_i=b\geq 2$ for all $i$ and $u=b-1$, that is, 
$$A=(a,Ha+dB)=\left(a,ba+d,b^2a+\frac{b^2-1}{b-1}d,\ldots,b^ka+\frac{b^k-1}{b-1}d\right),$$
the semigroup $\langle A\rangle$ is referred to as a \emph{\texttt{CNS} numerical semigroup}.
\end{dfn}

By the condition $\gcd(a, d) = 1$, it follows that $\gcd(a, h_1 a + d b_1) = \gcd(a, d b_1) = 1.$  
Moreover, we have $\gcd(A) = 1$. Hence, $\langle A \rangle$ is indeed a numerical semigroup.
In \cite{LiuSemigroupF}, the Frobenius problem for \texttt{GCNS} numerical semigroups is investigated.
This family of semigroups generalizes several important special cases, including Mersenne numerical semigroups \cite{Rosales2016}, repunit numerical semigroups \cite{Rosales.Repunit}, Thabit numerical semigroups \cite{Rosales2015}, Thabit numerical semigroups of the first kind base $b$ \cite{KyunghwanSong2020}, Cunningham numerical semigroups \cite{KyunghwanSong2020}, Proth numerical semigroups \cite{P.Srivastava}, as well as certain semigroups studied in \cite{GuZeTang,GuZe2020,KyunghwanSong}. For details, see \cite{LiuSemigroupF} for a comprehensive overview.

Let $\langle A\rangle$ be a \texttt{GCNS} numerical semigroup.
The main objective of this paper is to investigate the Frobenius problem for the quotient of $\langle A\rangle$ by $p$.
Throughout, we assume that $p$ is a positive divisor of $a$ and $p\neq a$.
This assumption excludes the trivial case $p = a$, since then $1\in \frac{\langle A \rangle}{a} = \mathbb{Z}^{+}$, which is immediate.

Our first main result characterizes the Ap\'ery set of the numerical semigroup \(\frac{\langle A\rangle}{p}\), which plays a key role in studying the Frobenius problem. (The Ap\'ery set will be formally introduced in Section \ref{Section-2}.)
This Ap\'ery set is closely related to the minimization problem defined by:
\begin{align*}
O_{B}^{H}(M)=\min\left\{uM+\sum_{i=1}^kx_i \bigm| \sum_{i=1}^k b_ix_i=M, \ x_i\in\mathbb{N}, 1\leq i\leq k\right\}.
\end{align*}
A detailed analysis of $O_B^H(M)$, based on the greedy algorithm, can be found in \cite{LiuSemigroupF}.

Inspired by the investigation of the Frobenius numbers for the quotients of numerical semigroups generated by almost arithmetic progressions in \cite{A.Adeniran,Liu-IJAC}, we study the Frobenius number and the genus of $\frac{\langle A \rangle}{p}$.
The second main result of this paper provides half-closed formulas for $F\left( \frac{\langle A \rangle}{p} \right)$ and $g\left( \frac{\langle A \rangle}{p} \right)$.
Given the values of $s_i$ for all $1 \leq i \leq k-1$, some closed-form expressions for $ F\left( \frac{\langle A \rangle}{p} \right)$ are derived.
When $p=1$, these formulas specialize to those for $F(\langle A \rangle)$ and $g(\langle A \rangle)$ of the classical \texttt{GCNS} numerical semigroup $\langle A \rangle$ established in \cite{LiuSemigroupF}.

All results in this paper were verified using the \emph{numericalsgps} package in GAP \cite{M.Delgado}.

The paper is organized as follows.
Section \ref{Section-2} presents a characterization of the Ap\'ery set for the numerical semigroup $\frac{\langle A\rangle}{p}$. 
In Section \ref{Section-3}, we derive formulas for the Frobenius number $F\left( \frac{\langle A \rangle}{p} \right)$ and the genus $g\left( \frac{\langle A \rangle}{p} \right)$.
Section \ref{Section-4} focuses on applications of the results from Section \ref{Section-3}. Given the values of $s_i$, we provide some closed-form formulas for $ F\left( \frac{\langle A \rangle}{p} \right)$.

\section{Ap\'ery Set}\label{Section-2}

We begin by recalling the important invariant known as the \emph{Ap\'ery set} of a numerical semigroup.
For convenience, we set $A := (a, B) = (a, b_1, b_2, \ldots, b_k)$ to facilitate the presentation of some conclusions.

Let $w \in \langle A\rangle \setminus \{0\}$.
The \emph{Ap\'ery set} of $w$ in $\langle A\rangle$ is defined as $Ape(A,w)=\{s\in \langle A\rangle \mid s-w\notin \langle A\rangle\}$. As shown in \cite{J.C.Rosales}, this set can be expressed as
$$Ape(A,w)=\{N_0,N_1,N_2,\ldots,N_{w-1}\},$$
where $N_r:=\min\{ a_0\mid a_0\equiv r\mod w, \ a_0\in \langle A\rangle\}$ for $0\leq r\leq w-1$. It is customary to take $w = a$.

The following results are due to Brauer and Shockley \cite{J. E. Shockley} and Selmer \cite{E. S. Selmer}, respectively.
\begin{lem}[\cite{J. E. Shockley}, \cite{E. S. Selmer}]\label{LiuXin001}
Let $A=(a, b_1, b_2, \ldots, b_k)$, and let $Ape(A,a)=\{N_0,N_1,N_2,\ldots,N_{a-1}\}$ be the Apéry set of $a$ in $\langle A \rangle$.
Then the Frobenius number and genus of $\langle A\rangle$ are given by
\begin{align*}
F(\langle A\rangle)=\max_{r\in \lbrace 0, 1, \ldots, a-1\rbrace}N_r -a\quad \text{and} \quad 
g(\langle A\rangle)=\frac{1}{a}\sum_{r=1}^{a-1}N_r-\frac{a-1}{2}.
\end{align*}
\end{lem}

When $w=a$ and $\gcd(a,d)=1$, the definition of $N_r$ implies
\begin{align*}
\{N_0, N_1, N_2,\ldots, N_{a-1}\}=\{N_{d\cdot 0}, N_{d\cdot 1}, N_{d\cdot 2},\ldots, N_{d\cdot (a-1)}\}.
\end{align*}

In \cite{LiuSemigroupF}, the computation of the Frobenius number and genus for \texttt{GCNS} numerical semigroups 
is closely related to the minimization problem defined by:
$$O_{B}^{H}(M):=\min\left\{\sum_{i=1}^kh_ix_i \mid \sum_{i=1}^k b_ix_i=M, \ x_i\in\mathbb{N}, 1\leq i\leq k\right\}.$$
This reduces to the $O_B(M)$ studied in \cite{Liu-Xin} when $H=(1,1,\dots,1)$.

Since the quotient of $\langle A\rangle$ by $p$ is a numerical semigroup and $\frac{a}{p} \in \frac{\langle A\rangle}{p}$, we let $N_{r,p}$ denote the element of the Ap\'ery set of $\frac{a}{p}$ in $\frac{\langle A\rangle}{p}$, where $0 \leq r \leq \frac{a}{p}-1$.

\begin{lem}\label{NDR-Quotient}
Let $A=(a, h_1a+db_1, \ldots, h_ka+db_k)$, where $a, k, h\in\mathbb{Z}^{+}$, $a\geq 2$, $d\in \mathbb{Z}\setminus \{0\}$, and $\gcd(A)=1$. If $d<0$, assume $h_ia+db_i>1$ for $1\leq i\leq k$.
Let $p$ be a positive divisor of $a$ (i.e., $p\mid a$ with $p\in \mathbb{Z}^{+}$). Then the Ap\'ery set of $\frac{a}{p}$ in the numerical semigroup $\frac{\langle A\rangle}{p}$ consists of the elements
\begin{align*}
N_{dr,p}=\min \left\{O_{B}^{H}(ma+rp) \cdot \frac{a}{p}+\left(\frac{ma}{p}+r\right)d \mid m\in \mathbb{N}\right\}\quad \text{for all}\quad 0\leq r\leq \frac{a}{p}-1.
\end{align*}
\end{lem}
\begin{proof}
By the definition of the Ap\'ery set, for any $0\leq r\leq \frac{a}{p}-1$, we have
\begin{align*}
N_{dr,p}&=\min\left\{ a_0\mid a_0\equiv dr\mod \frac{a}{p};\ a_0\in \frac{\langle A\rangle}{p}\right\}
=\min\left\{ a_0\mid a_0\equiv dr \mod \frac{a}{p};\ pa_0\in \langle A\rangle\right\}
\end{align*}
Since $p a_0 \in \langle A\rangle$, we can write
$$pa_0=a x_0+\sum_{i=1}^k(h_ia+db_i)x_i.$$
By minimality of $N_{dr,p}$, we must have $x_0 = 0$, and thus
\begin{align*}
N_{dr,p}&=\min\left\{\sum_{i=1}^k\frac{(h_ia+db_i)}{p}x_i\mid \sum_{i=1}^k\frac{(h_ia+db_i)}{p}x_i\equiv dr\mod \frac{a}{p}, \ x_i\in\mathbb{N}, 1\leq i\leq k \right\}
\\&=\min\left\{\left(\sum_{i=1}^k h_ix_i\right)\frac{a}{p}+\sum_{i=1}^k\frac{x_ib_id}{p} \mid \sum_{i=1}^k(h_ia+db_i)x_i\equiv drp \mod a, \ x_i\in\mathbb{N}, 1\leq i\leq k\right\}
\\&=\min\left\{\left(\sum_{i=1}^k h_ix_i\right)\frac{a}{p}+\sum_{i=1}^k\frac{x_ib_id}{p} \mid \sum_{i=1}^kdb_ix_i\equiv drp \mod a, \ x_i\in\mathbb{N}, 1\leq i\leq k\right\}
\end{align*}
Since $\gcd(A) = 1$, it follows that $\gcd(a, d) = 1$, and we deduce
\begin{align*}
N_{dr,p}&=\min\left\{\left(\sum_{i=1}^k h_ix_i\right)\frac{a}{p}+\sum_{i=1}^k\frac{x_ib_id}{p} \mid \sum_{i=1}^kb_ix_i\equiv rp \mod a, \ x_i\in\mathbb{N}, 1\leq i\leq k\right\}
\\&=\min\left\{\left(\sum_{i=1}^k h_ix_i\right)\frac{a}{p}+\left(\frac{ma}{p}+r\right)d \mid \sum_{i=1}^kb_ix_i=ma+rp, \ m,x_i\in\mathbb{N}, 1\leq i\leq k\right\}.
\end{align*}
For fixed $m$ (and hence fixed $M=ma+rp$), the expression $\sum_{i=1}^k h_i x_i$ is minimized by $O_B^{H}(ma +rp)$. This completes the proof.
\end{proof}

By Lemma \ref{NDR-Quotient}, we can define an intermediate function with respect to $m\in \mathbb{N}$ as follows:
$$N_{dr,p}(m):=O_{B}^{H}(ma+rp) \cdot \frac{a}{p}+\left(\frac{ma}{p}+r\right)d.$$
In our setting, it is not hard to show that $N_{dr,p}(m)$ is increasing with respect to $m$, which implies $N_{dr,p}=N_{dr,p}(0)$. This property facilitates further obtaining the formulas for the Frobenius number $F\left(\frac{\langle A\rangle}{p}\right)$ and the genus $g\left(\frac{\langle A\rangle}{p}\right)$.

Before proceeding, we introduce the following definition.  
Given a positive integer sequence  $B = (b_1, b_2, \ldots, b_k)$ with $1 = b_1 < b_2 < \cdots < b_k$ and $M \in \mathbb{N}$, define
\begin{align*}
opt_B(M)=O_B(M)=\min\left\{\sum_{i=1}^kx_i \mid \sum_{i=1}^kb_ix_i=M, \quad M,x_i\in \mathbb{N}, 1\leq i\leq k\right\}.
\end{align*}
This problem is known as the \emph{change-making problem} \cite{AnnAdamaszek}.

\begin{prop}{\em \cite[Lemma 2.6 and Proposition 2.7]{LiuSemigroupF}}\label{greedproper}
Let $B=\left(b_1,b_2,\ldots,b_k\right)$, where $b_1=1$, $b_{i+1}=s_ib_i+1$, and $s_i\geq s_{i-1}$ for $1\leq i\leq k-1$.
Then for any $M\in \mathbb{Z}^{+}$, the value $opt_B(M)=x_1+\dots+ x_k$ is uniquely determined by the following properties.
\begin{enumerate}
  \item[(1)] $x_k=\left\lfloor \frac{M}{b_k}\right\rfloor$.

  \item[(2)] $x_i\in \{0,1,\ldots,s_i\}$ for every $1\leq i\leq k-1$.

  \item[(3)] If $x_i=s_i$ for some $2\leq i\leq k-1$, then $x_1=\cdots =x_{i-1}=0$.
\end{enumerate}
\end{prop}

If a solution $X=(x_1,x_2,\ldots,x_k)$ of the equation $\sum_i x_ib_i=M$ satisfies conditions (1), (2), and (3) above, then $X=X(M)$ is called the \emph{greedy presentation} of $M$. The motivation for this terminology is explained in \cite{LiuSemigroupF}.
Define $R(M)=\{X(m): 0\leq m \leq M\}$, and introduce a colexicographic order on $R(M)$ as follows:
\begin{align*}
(x_1^{\prime},x_2^{\prime},\ldots,x_k^{\prime})\preceq (x_1,x_2,\ldots,x_k)
\Longleftrightarrow & x_i^{\prime}=x_i\ \ \text{for all}\ \ i>0;\ \ \text{or}
\\& x_j^{\prime}<x_j\ \ \text{for a certain}\ \ j>0\ \ \text{and}\ \ x_i^{\prime}=x_i\ \ \text{for all}\ \ i>j.
\end{align*}
Clearly, the order relation $\preceq$ is a total order on $R(M)$.

We now consider the $O_B^H(M)$, where the sequences $B$ and $H$ satisfy conditions \eqref{BBcondition} and \eqref{HHcondition}.
A direct computation yields 
\begin{align*}
O_{B}^{H}(M)=\min\left\{\sum_{i=1}^k(ub_i+1)x_i \mid \sum_{i=1}^k b_ix_i=M, \ x_i\in\mathbb{N}, 1\leq i\leq k\right\}
=uM+opt_B(M).
\end{align*}

Let $\langle A\rangle$ be a \texttt{GCNS} numerical semigroup. When $p$ is a positive divisor of $a$, we characterize the Ap\'ery set of $\frac{a}{p}$ in the numerical semigroup $\frac{\langle A\rangle}{p}$.

\begin{thm}\label{HBuadNDr}
Let $\langle A\rangle$ be a \texttt{GCNS} numerical semigroup. Let $p$ be a positive divisor of $a$.
If $ua+d+k-2\geq \sum_{i=1}^{k-1}s_i$, then $N_{dr,p}(m)$ is increasing in $m\in \mathbb{N}$ for all $0\leq r\leq \frac{a}{p}-1$. Specifically, if $X(rp)=(x_1,x_2,\ldots,x_k)$ is the greedy presentation of $rp$, then the Ap\'ery set of $\frac{a}{p}$ in the numerical semigroup $\frac{\langle A\rangle}{p}$ consists of the elements
\begin{equation}\label{Ndr-pre}
N_{dr,p}=\sum_{i=1}^k \frac{a x_i }{p}+r(ua+d)=\left(\sum_{i=1}^k(ub_i+1)x_i\right)\frac{a}{p}+rd
\quad \text{for all}\quad 0\leq r\leq \frac{a}{p}-1.
\end{equation}
\end{thm}
\begin{proof}
Recall that if $X(ma+rp)=(y_1,y_2,\ldots,y_k)$, then
$$N_{dr,p}(m)=\left(u(ma+rp)+\sum_{i=1}^k y_i\right)\frac{a}{p}+\left(\frac{ma}{p}+r\right)d.$$
Let $X((m+1)a+rp)=(z_1,\ldots,z_k)$. Note that $z_{k}\geq y_k$, and
\begin{align*}
N_{dr,p}(m+1)-N_{dr,p}(m)&=\frac{ua^2+ad}{p}+\left(\sum_{i=1}^k z_i-\sum_{i=1}^k y_i\right)\frac{a}{p}
\\& \geq\left(ua+d+\sum_{i=1}^{k-1}z_i-\sum_{i=1}^{k-1}y_i\right)\frac{a}{p}
\\(\textrm{by\ Proposition\ \ref{greedproper}})\quad  & \geq\left(ua+d-\sum_{i=1}^{k-1}s_i+k-2\right)\frac{a}{p}\geq 0.
\end{align*}
Thus $N_{dr,p}(m)$ is increasing so that $N_{dr,p}=N_{dr,p}(0)$. This completes the proof.
\end{proof}

\section{Frobenius Number and Genus}\label{Section-3}

In Equation \eqref{Ndr-pre} with $X(rp)=(x_1,x_2,\ldots,x_k)$, we refer to $w(rp)=\sum_{i=1}^k(ub_i+1)x_i$ as the \emph{weight} of $rp$. The following result holds.

\begin{lem}{\em \cite[Lemma 2.11]{LiuSemigroupF}}\label{colex-incre}
Let $M\in \mathbb{N}$. Suppose the sequences $B$ and $H$ satisfy conditions \eqref{BBcondition} and \eqref{HHcondition}, respectively, with $s_i\leq u+1$ for all $1\leq i\leq k-1$. Let $X(r_1p)=(x_1^{\prime},\ldots,x_k^{\prime})\in R(M)$ and $X(r_2p)=(x_1,\ldots,x_k)\in R(M)$. If $(x_1^{\prime},\ldots,x_k^{\prime})\preceq (x_1,\ldots,x_k)$, then $w(r_1p)\leq w(r_2p)$. Moreover, if $r_1\neq r_2$ and $s_i<u+1$, then $w(r_1p)<w(r_2p)$.
\end{lem}

We now present the main result of this section.

\begin{thm}\label{Quotient-FrobeniusN}
Let $\langle A\rangle$ be a \texttt{GCNS} numerical semigroup. Let $p$ be a positive divisor of $a$ and $p\neq a$.
Suppose $X(a-p)=(x_1,x_2,\ldots,x_k)$ is the greedy presentation of $a-p$, and note that $x_k=\left\lfloor \frac{a-p}{b_k}\right\rfloor$. If $ua+d+k-2\geq \sum_{i=1}^{k-1}s_i$, $a+pd\geq 0$, and $s_i< u+1$ for all $1\leq i\leq k-1$, then 
\begin{align*}
F\left(\frac{\langle A\rangle}{p}\right)=\sum_{i=1}^kx_i\cdot \frac{a}{p}+\left(\frac{a}{p}-1\right)(ua+d)-\frac{a}{p}.
\end{align*}
In particular, if $d$ is positive, it suffices to assume $ua+d+k-2\geq \sum_{i=1}^{k-1}s_i$ and $s_i\leq u+1$ for $1\leq i\leq k-1$. The above formula remains valid.
\end{thm}
\begin{proof}
By Theorem \ref{HBuadNDr}, Lemma \ref{colex-incre}, and Equation \eqref{Ndr-pre}, under the condition $s_i<u+1$, we have
$$N_{d(r+1),p}-N_{dr,p}=(w((r+1)p)-w(rp))\frac{a}{p}+d\geq \frac{a}{p}+d\geq0.$$
Therefore, we obtain
$$\max_{0\leq r\leq \frac{a}{p}-1} N_{dr,p}=N_{d(\frac{a}{p}-1),p}=\sum_{i=1}^kx_i\cdot \frac{a}{p}+\left(\frac{a}{p}-1\right)(ua+d).$$
Applying Lemma \ref{LiuXin001} yields the formula for $F\left( \frac{\langle A \rangle}{p} \right)$. When $d > 0$, the weaker condition $s_i \leq u + 1$ suffices by Lemma \ref{colex-incre}. This completes the proof.
\end{proof}

\begin{thm}\label{Quotient-Genusp}
Let $\langle A\rangle$ be a \texttt{GCNS} numerical semigroup. Let $p$ be a positive divisor of $a$ and $p\neq a$.
If $ua+d+k-2\geq \sum_{i=1}^{k-1}s_i$, then
\begin{align*}
g\left(\frac{\langle A\rangle}{p}\right)
=\sum_{r=1}^{\frac{a}{p}-1}\left(\sum_{i=1}^kx_i\right)_{rp}+\frac{(ua+d-1)}{2}\left(\frac{a}{p}-1\right),
\end{align*}
where $\left(\sum_{i=1}^{k}x_i\right)_{rp}$ denotes the sum of elements in the greedy presentation of $rp$. and $x_k=\left\lfloor \frac{rp}{b_k}\right\rfloor$.
\end{thm}
\begin{proof}
The result follows directly from Equation \eqref{Ndr-pre} and Lemma \ref{LiuXin001}.
\end{proof}

When $p=1$, Theorems \ref{Quotient-FrobeniusN} and \ref{Quotient-Genusp} reduce to \cite[Thoerems 2.12 and 2.13]{LiuSemigroupF}.
The following result can be verified using the \emph{numericalsgps} package in GAP \cite{M.Delgado}.

\begin{exa}
Let $B=(1,3,7,22)$ and $u=4$. Then $s_1=2,s_2=2,s_3=3$, and $H=(5,13,29,89)$. Therefore, we obtain
$$A=(a,5a+d,13a+3d,29a+7d,89a+22d).$$ 
For $a=50$ and $p=5$, the greedy presentation of $a-p$ is $(1,0,0,2)$.
Assuming $\gcd(a,d)=1$, we obtain
$$F\left(\frac{\langle A\rangle}{5}\right)=1820+9d.$$
For instance, when $d=1$, the Frobenius number $F\left(\frac{\langle A\rangle}{5}\right)$ is $1829$; when $d=3$, 
the Frobenius number $F\left(\frac{\langle A\rangle}{5}\right)$ is $1847$.
The genus can be computed similarly.
\end{exa}

\begin{exa}
Let $B=(1,4,17,69)$ and $u=3$. Then $s_1=3,s_2=4,s_3=4$, and $H=(4,13,52,208)$. Therefore, we obtain
$$A=(a,4a+d,13a+4d,52a+17d,208a+69d).$$ 
For $a=243$ and $p=9$, the greedy presentation of $a-p$ is $(2,2,1,3)$.
Assuming $\gcd(a,d)=1$, we obtain
$$F\left(\frac{\langle A\rangle}{9}\right)=19143+26d.$$
For instance, when $d=2$, the Frobenius number $F\left(\frac{\langle A\rangle}{9}\right)$ is $19195$.
\end{exa}

\section{Applications}\label{Section-4}

For a \texttt{GCNS} numerical semigroup, given the value of $s_i$ for all $1\leq i\leq k-1$, a closed-form formula for $F\left(\frac{\langle A\rangle}{p}\right)$ is obtained.

\begin{prop}
Let $A=(a, (u+1)a+d, ((s+1)u+1)a+(s+1)d)$, where $d\in \mathbb{Z}^{+}$, $\gcd(a,d)=1$, $a\geq 2$, and $u+1\geq s\geq 1$. 
Let $p$ be a positive divisor of $a$ and $p\neq a$. Then the Frobenius number is given by
\begin{align*}
&F\left(\frac{\langle A\rangle}{p}\right)=\left(a-p-s\left\lfloor \frac{a-p}{s+1}\right\rfloor\right)\frac{a}{p}+\left(\frac{a}{p}-1\right)(ua+d)-\frac{a}{p},
\end{align*}
and the genus by
\begin{align*}
&g\left(\frac{\langle A\rangle}{p}\right)=\left(\frac{a}{p}-1\right)\cdot \frac{ua+a+d-1}{2}-s\sum_{r=1}^{\frac{a}{p}-1} \left\lfloor \frac{rp}{s+1}\right\rfloor.
\end{align*}
\end{prop}
\begin{proof}
Let $X((\frac{a}{p}-1)p)=(x_1,x_2)$ be the greedy presentation of $(\frac{a}{p}-1)p$, i.e., $a-p=(s+1)x_2+x_1$. Then
$$x_1+x_2=\left\lfloor \frac{a-p}{s+1}\right\rfloor +a-p-(s+1)\left\lfloor \frac{a-p}{s+1}\right\rfloor=a-p-s \left\lfloor \frac{a-p}{s+1}\right\rfloor.$$
The formula for the Frobenius number follows from Theorem \ref{Quotient-FrobeniusN}, and the genus formula from Theorem \ref{Quotient-Genusp}.
\end{proof}

\begin{prop}
Let $A=(a, (u+1)a+d, (3u+1)a+3d, (10u+1)a+10d)$, where $d\in \mathbb{Z}^{+}$, $\gcd(a,d)=1$, $u\geq 2$, and $a\geq 2$. 
Let $p$ be a positive divisor of $a$ and $p\neq a$. Then the Frobenius number is given by
$$F\left(\frac{\langle A\rangle}{p}\right)=\frac{a}{p}\left(\left\lfloor \frac{a-p}{10}\right\rfloor +ua+d+\varphi_1((a-p) \mod 10)\right)-(ua+d),$$
where $(\varphi_1(i))_{0\leq i \leq 9} =(-1,0,1,0,1,2,1,2,3,2)$.
\end{prop}
\begin{proof}
The result follows directly from Theorem \ref{Quotient-FrobeniusN} with $s_1=2,s_2=3$ by routine computation.
\end{proof}

\begin{prop}
Let $A=(a, (u+1)a+d, (4u+1)a+4d, (17u+1)a+17d)$, where $d\in \mathbb{Z}^{+}$, $\gcd(a,d)=1$, $u\geq 3$, and $a\geq 2$. 
Let $p$ be a positive divisor of $a$ and $p\neq a$. Then the Frobenius number is given by
$$F\left(\frac{\langle A\rangle}{p}\right)=\frac{a}{p}\left(\left\lfloor \frac{a-p}{17}\right\rfloor +ua+d+\varphi_2((a-p) \mod 17)\right)-(ua+d),$$
where $(\varphi_2(i))_{0\leq i \leq 16} =(-1,0,1,2,0,1,2,3,1,2,3,4,2,3,4,5,3)$.
\end{prop}
\begin{proof}
The result follows directly from Theorem \ref{Quotient-FrobeniusN} with $s_1=3,s_2=4$ by routine computation.
\end{proof}

\begin{prop}
Let $A=(a, (u+1)a+d, (3u+1)a+3d, (7u+1)a+7d, (29u+1)a+29d)$, where $d\in \mathbb{Z}^{+}$, $\gcd(a,d)=1$, $u\geq 3$, and $a\geq 2$. 
Let $p$ be a positive divisor of $a$ and $p\neq a$. Then the Frobenius number is given by
$$F\left(\frac{\langle A\rangle}{p}\right)=\frac{a}{p}\left(\left\lfloor \frac{a-p}{29}\right\rfloor +ua+d+\varphi_3((a-p) \mod 29)\right)-(ua+d),$$
where $(\varphi_3(i))_{0\leq i \leq 28} =(-1,0,1,0,1,2,1,0,1,2,1,2,3,2,1,2,3,2,3,4,3,2,3,4,3,4,5,4,3)$.
\end{prop}
\begin{proof}
The result follows directly from Theorem \ref{Quotient-FrobeniusN} with $s_1=2,s_2=2,s_3=4$ by routine computation.
\end{proof}

For given values of $s_i$ for all $1\leq i\leq k-1$, similar formulas can be derived.
From Theorems \ref{Quotient-FrobeniusN} and \ref{Quotient-Genusp}, we directly obtain the following results.
\begin{cor}\label{Quotient-CNS-Frobenius}
Let $\langle A\rangle$ be a \texttt{CNS} numerical semigroup with $d\in\mathbb{Z}^{+}$. Let $p$ be a positive divisor of $a$ and $p\neq a$.
Suppose $X(a-p)=(x_1,x_2,\ldots,x_k)$ is the greedy presentation of $a-p$.
If $a\geq k-1-\frac{d-1}{b-1}$, then
\begin{align*}
F\left(\frac{\langle A\rangle}{p}\right)= \sum_{i=1}^kx_i \cdot \frac{a}{p} +\left(\frac{a}{p}-1\right)((b-1)a+d)-\frac{a}{p}.
\end{align*}
\end{cor}
\begin{proof}
Since $d > 0$, the conclusion follows by applying Theorem \ref{Quotient-FrobeniusN} with $s_1=s_2=\cdots=s_{k-1}=b$ and $u=b-1$.
\end{proof}

\begin{cor}\label{Quotient-CNS-Genus}
Let $\langle A\rangle$ be a \texttt{CNS} numerical semigroup with $d\in\mathbb{Z}^{+}$. Let $p$ be a positive divisor of $a$ and $p\neq a$.
If $a\geq k-1-\frac{d-1}{b-1}$, then
\begin{align*}
g\left(\frac{\langle A\rangle}{p}\right)
=\sum_{r=1}^{\frac{a}{p}-1}\left(\sum_{i=1}^kx_i\right)_{rp}+\frac{((b-1)a+d-1)}{2}\left(\frac{a}{p}-1\right),
\end{align*}
where $\left(\sum_{i=1}^{k}x_i\right)_{rp}$ denotes the sum of elements in the greedy presentation of $rp$.
\end{cor}
\begin{proof}
The result follows by applying Theorem \ref{Quotient-Genusp} to the case where $s_1=s_2=\cdots =s_{k-1}=b$ and $u=b-1$.
\end{proof}

When $p=1$, Corollaries \ref{Quotient-CNS-Frobenius} and \ref{Quotient-CNS-Genus} specialize to \cite[Corollaries 2.16 and 2.17]{LiuSemigroupF}.








\noindent
{\small \textbf{Acknowledgements:}}
The author thanks his advisor Guoce Xin for guidance and support.

\end{document}